\documentclass[12]{amsart}

\usepackage{tabularx, verbatim, amsfonts, amsthm, amssymb, color, epsfig, multienum, setspace}
\usepackage[margin=3cm]{geometry}
\newtheorem{lemma}{Lemma}

\usepackage[justification=centering]{caption}
\usepackage{graphicx}

\usepackage{tikz}

\newtheorem{theorem}{Theorem}[section]

\newtheorem{corollary}[theorem]{Corollary}
\theoremstyle{definition}

\newtheorem{definition}[theorem]{Definition}

\newtheorem{remark}[theorem]{Remark}
\newtheorem{example}[theorem]{Example}

\newtheorem{proposition}[theorem]{Proposition}

\numberwithin{equation}{section}

\newcommand{\kthree}[7]{
\begin{tikzpicture}[scale=#1,auto,main node/.style={circle,draw,font=\sffamily\Large\bfseries}]
	\tikzstyle{every node}=[draw, shape=rectangle];
	\path (0,8.66)			node (g1) {$#2$};
	\path (5,0)		node (g2) {$#3$};
	\path (-5,0)		node (g3) {$#4$};
	\draw (g1) -- (g2);
	\draw (g2) -- (g3);
	\draw (g3) -- (g1);

  \path[every node/.style={font=\sffamily}]
    (g1) edge node {$#5$} (g2)
    (g2) edge node {$#6$} (g3)
    (g3) edge node {$#7$} (g1);
\end{tikzpicture}
}

\newcommand{\R}{\mathbb{R}}

\newcommand{\lcm}{\mathrm{lcm}}

\begin{document}

\title{Determinantal Conditions for Modules of Generalized Splines}
\author[Calta]{Kariane Calta}
\address{Department of Mathematics, Vassar College, Poughkeepsie, NY, 12604}

\author[Rose]{Lauren L. Rose}
\address{Department of Mathematics, Bard College, Annandale-on-Hudson, NY, 12504}
\email{rose@bard.edu}

\begin{abstract}
  Generalized splines on a graph $G$ with edge labels in a commutative ring $R$ are vertex labelings such that if two vertices share an edge in $G$, the difference between the vertex labels lies in the ideal generated by the edge label. When $R$ is an integral domain, the set of all such splines is a finitely generated $R$-module $R_G$ of rank $n$, the number of vertices of $G$. We find determinantal conditions on subsets of $R_G$ that determine  whether $R_G$ is a free module, and if so, whether a so called "flow-up class basis" exists.  
\end{abstract}

\maketitle
\section{Introduction} 

Rings, modules and vector spaces of polynomial splines have been well studied in terms of their algebraic properties, for example in \cite{billera1}, \cite{billera2}, \cite{haas}, \cite{rose1}, \cite{rose2}.  The notion of a polynomial spline can be generalized to define splines on edge labeled graphs over arbitrary rings. The study of generalized splines over rings was introduced by Gilbert, Polster, and Tymoczko in \cite{julia}.  Generalized integer splines were introduced in  \cite{smith students}, where they constructed module bases for integer splines on cycles. 

If $R$ is a ring and $G$ is a graph with $n$ vertices, then the set of generalized splines $R_G$ is a submodule of $R^n$. When $R_G$ is a free $R$-module, the bases of particular interest are the \textit{flow-up class bases}.  Elements of a flow-up class basis are $n$-tuples with $k$ leading zeros, where $0 \leq k \leq n-1$. As a result, if we order the basis elements according to the number of leading zeroes, they form the columns of a lower triangular matrix, whose determinant is easy to compute.

In this paper we address two questions about modules of generalized splines:

\begin{enumerate}
    \item For which rings do flow-up class bases for arbitrary graphs always exist? 

\smallskip

\item  If a flow-up class basis doesn't exist but the module $R_G$ is free, do we have any information about bases?  
\end{enumerate}

As an answer to Question 1, we generalize results from \cite{studentpaper} and show that flow-up class bases for spline modules necessarily exist over principal ideal domains. In particular, this includes Euclidean rings such as the integers and polynomial rings in one variable over a field as well as non-Euclidean rings such as $\mathbb{Z}[1/2(1 + \sqrt{-19})]$. (\cite{noneuclidean}) We also find that for a fixed ordering of the  vertices of the graph $G$, the leading term of a basis element is unique up to multiplication by a unit. We then give an example to show that when $R$ is not a principal ideal domain, flow-up class bases don't necessarily exist, so our result is essentially the best possible.

\

In answer to Question 2, we first generalize results of \cite{gjoni}, \cite{mahdavi}, \cite{blaine},\cite{alt1}, and \cite{alt2} to show that over a principal ideal domain, a set of $n$ elements $\mathcal{B}=\left\{ B_1, \ldots, B_n \right\}$ is a basis for $R_G$ if and only the determinant of the matrix $[B_1,\ldots,B_n]$  has a prescribed form. We further show that over GCD domains (where flow-up class bases need not exist), any set of elements with the same prescribed form will be a basis. We prove the converse under certain conditions, but we conjecture that our result is true in all cases.

\

In \cite{multivariatesplines}, Rose proved the following theorem for modules $M_r$ of $r$-differentiable polynomial splines defined on $D$, a polyhedral subdivision of $\mathbb{R}^d$: 

\begin{theorem}  Let $B=\left\{ B_, \dots, B_n \right\} \subset M_r$. Then $B$ is a basis of $M_r$ if and only if $\det[B_1, \ldots B_n] = Q^{r+1}$, where $Q$ is the product of the linear forms that define the co-dimension $1$ interior faces of $D$. \end{theorem} 

In this setting, the linear forms are either relatively prime or equal, and correspond to edge labels in the dual graph of $D$. In the case of generalized polynomial splines, the situation is more complicated if you allow arbitrary polynomials as edge labels. However, with a few restrictions on the edge labels, we prove an analogous and more general result for arbitrary graphs over a GCD domain:

\begin{theorem} Let $R$ be a GCD domain and let $G$ be a graph with edge labels $a_1, \ldots, a_m \in R$  that are pairwise relatively prime. Then $R_G$ is free with bases $\mathcal{B}=\left\{ B_1, \ldots, B_n \right\}$ if and only if $\det(B)=uQ$ for some unit $u \in R$. \end{theorem} 



\section{Preliminaries}

We begin by defining a generalized spline over a commutative ring with unity. Some of our results will require stronger assumptions about $R$, e.g. that $R$ is an integral domain, principal ideal domain, or a GCD domain.  

Let $R$ be a commutative ring with unity and let $G$ be a connected graph with $n$ vertices, edge set $E = \left\{e_1,\ldots, e_m\right\}$, and $A = \left\{ a_1,\ldots, a_m\right\}$ an assignment of edge labels in $R$.

\begin{definition}
A \textbf{generalized spline} on $G$ over the ring $R$ is an assignment of an $n$-tuple  $F =(f_1, \ldots f_n)\in R^n$ to the vertices of $G$ such that for any two vertices $v_i$ and $v_j$ that are connected by an edge $e_{k}$, we have $f_i\equiv f_j \bmod{a_{k}}$. That is, $f_i- f_j \in <a_{k}>$, the ideal in $R$ generated by $a_{k}$.  
\end{definition}

We denote the set of all $R$-splines on $(G,A)$, the graph $G$ with edge weights $A$, by $R_{(G,A)}$ or just $R_G$ when $A$ is fixed. The following basic properties are easily verified.   

\begin{remark} \leavevmode
\begin{enumerate}
    \item For any $G$, $R_G$ always contains $\mathbf{0} = (0,0,\ldots ,0)$ and  $\mathbf{1}= (1,1,\ldots ,1)$, since for all $a \in R$,  $0  \equiv    0 \bmod a$  and  $ 1   \equiv     1 \bmod a$.
    \item The splines $R_{G}$ form a module over $\mathbb{Z}$ under componentwise addition and scalar multiplication. 
    \item When $R$ a principal ideal domain, $R_G$ is a free module, and hence has a module basis. For other rings, freeness is not guaranteed.
\end{enumerate}
\end{remark} 

When $R_G$ is free, the bases that are of particular interest to us are the \textbf{flow-up class bases}, primarily because when placed as rows or columns of a matrix, the matrix will be triangular, making it easy to compute the determinant. In fact, the initial non-zero components of 
such basis elements are unique up to multiplication by a unit, and are important to understanding the structure of $R_G$.  

\begin{definition}\leavevmode
\begin{enumerate}
    \item For $0 \leq i < n$, let $\mathcal{F}_i= \left\{F=(0,0,\ldots,f_{i+1},\ldots,f_n): F \in R_G \right\}$. 
    \item We define $\mathcal{F}_{n} = \left\{\mathbf{0}\right\}$, the set containing just the zero spline $\mathbf{0} = (0,0,\ldots, 0)$.
    \item We refer to $\mathcal{F}_i$ as the $i$th \textbf{flow-up class}.
    \item A \textbf{flow-up class basis} for $R_G$ is a module basis consisting of an element from each non-zero flow-up class. 
\end{enumerate} 
\end{definition}

\begin{remark}\leavevmode 
\begin{enumerate}
\item By definition, the $\mathcal{F}_i$ form a partition of $R_G$. 
\item Our definition of  $\mathcal{F}_{i}$ is stricter than that used by \cite{julia}, which only requires that the first $i$ components be zero. In their case, the sets $\mathcal{F}_i$ are nested, rather than forming a partition.  
\end{enumerate}
\end{remark}

\begin{example} 
Consider the graph $(C_3,A)$ depicted in Figure \ref{figure of GA weighted},   where $A =(a_1,a_2,a_3)$. A set of vertex labels $(f_1,f_2,f_3)$ will form a spline if and only if the following conditions hold:

\begin{align*}
a_1| f_1 - f_2\\
a_2|f_2-f_3\\
a_3| f_3 - f_1
\end{align*}
\end{example}

\begin{figure}
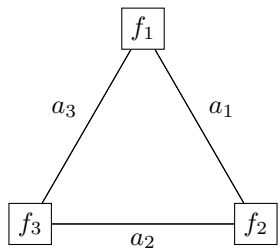

$\kthree{0.3}{f_{1}}{f_{2}}{f_{3}}{a_1}{a_2}{a_3}$
\caption{Generalized spline on a 3-cycle with edge weights $a_1,a_2,a_3$}
\label{figure of GA weighted}
\end{figure}

Next, we give a specific example of an integer spline. 

\begin{example} 
In Figure \ref{figure of GA with integers} , the edge labels of $C_3$ are $a_1=4$, $a_2=5$ and $a_3=1$. It is easy to verify that the vertex labeling $(3,15,5)$ is an integer spline on $C_3.$
\end{example} 

\begin{figure}
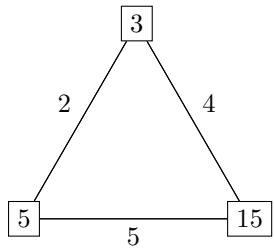

$\kthree{0.3}{3}{15}{5}{4}{5}{2}$
\caption{Example of a Spline on a 3-cycle}
\label{figure of GA with integers}
\end{figure}

\section{Flow-up Class Bases}

We begin by showing that if $R$ is an integral domain and $G$ is an arbitrary graph, the flow-up classes in $R_G$ are non-empty. 

\begin{lemma}
Let $R$ be an integral domain and let $(G,A)$ be an edge labeled graph with $n$ vertices. Then $\mathcal{F}_i \subset R_G$ is non-empty for all $0 \leq i \leq n$. \end{lemma}

\begin{proof}
First note that $\mathbf{0}=(0,\ldots,0)  \in \mathcal{F}_{n}$ and $\mathbf{1}=(1,\ldots,1) \in \mathcal{F}_{0}$, so they are non-empty.  For $0 < i < n$, let $F =(f_1, \ldots, f_n)\in \R^n$ be defined as follows. Let  $f_{j}=0$ if $j \neq i+1$ and let $f_{i+1}$ be the product of the edge labels incident to vertex $v_{i+1}$.  Consequently $F$ has the form $F =(0, \dots,0, f_{i+1},0\ldots,0)$, with $f_{i+1} = a_{{i}_{1}} a_{{i}_{2}} \cdots a_{{i}_{k}}$ where $e_{{i}_{1}}, e_{{i}_{2}}, \ldots, e_{{i}_{k}}$ are the edges incident to vertex $v_{i + 1}$.  

Next, we show that $F$ is a spline.  If  edge $e_p$ is is not incident to vertex $i+1$ then the defining equation for that edge will be  $0 \equiv 0 \bmod a_p$, which is always true.  For edges $e_{{i}_r}$ incident to vertex $v_{i+1}$, the defining equation will be  $f_{i+1}\equiv 0 \bmod a_{{i}_r}$.  Since $f_{i+1} = a_{{i}_1} a_{{i}_2} \cdots a_{{i}_k}$, we have that $a_{{i}_r} | f_{i+1}$ for each $1 \leq r  \leq k$, so the equation holds.  Since $F$ satisfies the spline equations for each edge, $F \in R_G$ and since $F$ has exactly $i$ leading zeros, we have that $F \in \mathcal{F}_i$. Hence for each $0 \leq i \leq n$, $\mathcal{F}_i$ is non-empty.
\end{proof}

We define the leading term of a non-zero spline $F$ to be the first non-zero entry of $F$.

\begin{definition}
Let $F \in \mathcal{F}_{i}$,  where $1 \leq i < n$, so $F = (0, 0, \ldots, f_{i+1}, \ldots, f_{n})$.  We define the \textbf{leading term} of $F$,  $LT(F) = f_{i+1}$, the first non-zero entry of $F$.
\end{definition}

Let $L(\mathcal{F}_i)= \left\{ LT(F): F \in \mathcal{F}_i \right\}$.  

\begin{lemma} The set of all leading terms of splines in $\mathcal{F}_i$, $L(\mathcal{F}_i)$, is an ideal of $R$. \end{lemma}

\begin{proof} 
 Let $F$ and $H$ be splines in $\mathcal{F}_i$. Then $LT(F) + LT(H)$ is either the leading term of an element of $\mathcal{F}_i$ or it's zero, hence it is in $L(\mathcal{F}_i)$.   If $r \in R$, and $r \neq 0$, then $rLT(F)= LT(rF) \in L(\mathcal{F}_i)$. Thus $L(\mathcal{F}_i)$ is an ideal of $R$. 
\end{proof}

When $R$ is a PID, we can define the notion of a \textbf{minimal element} of a flow-up class.

\begin{definition} Let $R$ be a PID. A spline $B \in \mathcal{F}_i$ is called a \textbf{minimal element} of $\mathcal{F}_i$ if $LT(B)$ is a generator of the ideal $L(\mathcal{F}_i)$.
\end{definition}

Note that minimal elements are not in general unique, but their leading terms are unique up to multiplication by a unit. 

We now show that when we take a minimal element from each flow-up class $\mathcal{F}_i$, where $0 \leq i < n$, we will always get a flow-up class basis for $R_G$.  First, we need a lemma. 

\begin{lemma}\label{lem:flowup}
Let $F \in \mathcal{F}_i$, and let $B$ be a minimal element of $\mathcal{F}_i$. There exists $r \in R$ such that $F-rB\in \mathcal{F}_k$ for some $k>i$.
\end{lemma}

\begin{proof}
Since $LT(F) \in <LT(B)>$, there exists some $r \in R$ such that $LT(F)=rLT(B)$. Let $J=F-rB$. Then the $i+1$st component of $J$ will be $0$, so $J$ has at least $i+1$ leading zeros. Thus $J\in \mathcal{F}_k$ for some $k>i$. 
\end{proof}

\begin{theorem} \label{fcbasis}
Let $R$ be a PID and $(G,A)$ be an edge-labeled graph with $n$ vertices. Then $R_G$ is a free $R$-module, and it has a flow-up class basis. Moreover, if  $\mathcal{B}=\{ B_1, \dots, B_n\} \subset R_G$ where each $B_i \in \mathcal{F}_i$ is minimal, then $\mathcal{B}$ is a flow-up class basis. 
\end{theorem}

\begin{proof}
By the structure theorem for modules over a PID, $R_G$ is free for all $G$. Let $B_i \in \mathcal{F}_i$ be such that $LT(B_i)$ generates $<L(\mathcal{F}_i)>$. Note that the matrix $\begin{bmatrix} B_1 &\cdots &B_{n} \end{bmatrix}$ is a lower triangular matrix where the diagonal entries are the leading terms of the $B_i$'s. Since the leading terms are each non-zero, the determinant will be non-zero. It follows that $\mathcal{B}$ is linearly independent over $R$. We will show by induction that for every $H \in R_G$, $H \in span(\mathcal{B}_i)$. Note that the only element of $\mathcal{F}_n$, the zero spline, is trivially in the span.    

Suppose $k > 0$ and that whenever  $0 \leq j < k$, $\mathcal{F}_{n-j} \subset span(\mathcal{B})$.  Let $H \in \mathcal{F}_{n-k}$.  By Lemma \ref{lem:flowup}, there exists some $q \in R$ and $i$ such that $n-k<n-i\leq n$ so that $F = H-qB_{n-k}  \in \mathcal{F}_{n-i}$.  Thus  $H=F+qB_i$ and $0 \leq i < k$.  By the induction hypothesis, $F \in span(\mathcal{B})$, so $H \in span(\mathcal{B})$. By induction we have shown that any for any $S \in R_G$, $S \in span(\mathcal{B})$. Since $\mathcal{B}$ is a linearly independent set that spans $R_G$, it is a basis.
\end{proof}

Note that the only property of splines that we used in this theorem is that $R_G$ is a submodule of $R^n$ with non-empty flow-up classes. Thus, this theorem can be stated more generally. If some of the flow-up classes are empty, the result will still hold, i.e., submodules of spline modules will also have flow-up class bases. In fact, when $R$ is a Euclidean domain we can construct the leading terms explicitly in terms of the edge labels. See \cite{jefflauren} for details. 

\section{Determinants} 

When $R$ is an integral domain and $M$ is a finitely generated submodule of $R^n$, we have a nice criterion for when an arbitrary set of $n$ elements will form a basis.

Throughout this section, we will need to refer to the determinant of a matrix whose columns are elements of a subset of $M$. To simplify notation, we make the following definition:

\begin{definition}
Let  $C=\left\{C_1, \ldots C_n \right\}$ be an ordered subset of $M$. We define 
$\det(C)=det \begin{bmatrix} C_1,\ldots,C_n \end{bmatrix}$. \end{definition}

\begin{theorem}\label{basis-theorem} Let $R$ be an integral domain, $M$ a finitely generated rank $n$ submodule of $R^n$, and suppose that $M$ is free with basis $\mathcal{B}=\left\{B_1,.,.B_n\right\}$. Then  $\mathcal{C}=\left\{C_1, \ldots, C_n\right\} \subset M$ is another basis for $M$ if and only if $\det(\mathcal{C}) = u\det(\mathcal{B})$, where u is a unit in $R$. \end{theorem}
	
\begin{proof}

First assume that $\mathcal{C}=\left\{ C_1, \cdots, C_n \right\}$ is a basis for $R^n$. Since $\mathcal{B}$ is also a basis for $R^n$, for each $C_i \in \mathcal{C}$, there exist $a_{ij} \in R$ with $1\leq i,j \leq n$ so that $C_i = \sum_{j=1}^n a_{ij}B_j$. If we define $A= \begin{bmatrix} a_{ij} \end{bmatrix}$, the matrix with $ij$th entry equal to $a_ij$, $C=\begin{bmatrix} C_1,\ldots, C_n \end{bmatrix}$ and $B= \begin{bmatrix} B_1,\ldots,B_n \end{bmatrix}$ then $C=AB$. Similarly, since $\mathcal{C}$ is a basis for $R^n$, there exists a matrix $A'$ so that $B=A'C$. 
Then $C=AB=AA'C$. Since $\mathcal{C}$ is a basis, $\det(\mathcal{C}) \neq 0$,  hence $\det(A)\det(A')=1$. This implies that $\det(A)$ is a unit in $R$ and since $\det(C)=\det(A) \det (B)$, the result follows. 

Now assume that $\det(C)=u\det(B)$ for some unit $u \in R$. 
Since $\mathcal{B}$ is a basis for $R^n$, there exists some matrix $A$ so that $C=AB$.  Then $\det(C)=\det(A)\det(B)$, and since $\det(C)=u\det(B)$, we have that $u\det(B)=\det(C)=\det(A)\det(B)$. Since $\mathcal{B}$ is a bais, $\det(B) \neq 0$, and it follows that $\det(A)$ is a unit. Thus there exists some matrix $A'$ so that $A'A=I$ and so $A'C=A'AB=B$ so $A'C=B$. Since $\mathcal{B}$ is a basis, so is $\mathcal{C}$.  
\end{proof} 

We can now interpret this theorem in terms of spline modules. 

\begin{corollary}  \label{pid} Let $R$ be a PID and let $(G,A)$ be an arbitrary edg-labeled graph with $n$ vertices. Furthermore, suppose that $\mathcal{B}=\left\{B_1,.,.B_n\right\}$ is a flow-up class basis for $R_G$. For each $i$, let $l_i$ be the leading term of $B_i$ and define $Q=l_1\cdots l_n$. Then $\mathcal{C}=\left\{C_1,.,.C_n\right\}$ is a basis for $R_G$ if and only if $\det(\mathcal{C}) = uQ$ for some unit $u \in R$.
\end{corollary}

\begin{proof}
First note that a flow-up class basis for $R_G$ is guaranteed to exist by Theorem \ref{fcbasis}. If $\mathcal{B}$ is a flow-up class basis with leading terms $l_i$, then $\det(\mathcal{B})=l_1 \cdots l_n=Q$. Let $\mathcal{C}=\left\{C_1, \ldots, C_n \right\}$. By Theorem \ref{basis-theorem}, $\mathcal{C}$ is a basis if and only $\det(\mathcal{C})=u\det(\mathcal{B})=uQ$ for some unit $u \in R$. 

\end{proof}

We would like similar criteria to determine whether a set of elements form a basis of $R_G$, even when no flow-up class basis exists. For example, we will prove later that if $R$ is not a PID, flow-up class bases need not exist, even if $R_G$ is free.

For splines over integral domains, we provide a theoretical framework for such criteria, and for GCD domains over graphs with relatively prime edge labels, we provide an explicit construction of what $Q$ needs to be in this case.

\begin{lemma} \label{Cramer} Let $R$ be an integral domain. Suppose  $\det \begin{bmatrix} B_1, \ldots, B_n \end{bmatrix} = Q$, where $B_1, \ldots, B_n \in R_G$. Then $QR^n \subset  span\left\{B_1,\ldots,B_n\right\}$.
\end{lemma}

\begin{proof} Let $B=[B_1,\ldots,B_n]$. If $B_1, \cdots B_n$ are not linearly independent, then $Q=det(B)=0$ and so $QR^n=\textbf{0} \in span \left\{B_1,\ldots,B_n\right\}$. Now suppose that the $B_i$ are linearly independent and let $D \in R_G$. We will show that there exists some $x_1,\ldots, x_n \in R$ such that:

\begin{equation*}
\begin{pmatrix}
b_{11} & b_{12} & \cdots & b_{1n} \\
b_{21} & b_{22} & \cdots & b_{2n} \\
\vdots  & \vdots  & \ddots & \vdots  \\
b_{n1} & b_{n2} & \cdots & b_{nn} 
\end{pmatrix} \begin{pmatrix} x_1 \\ x_2 \\ \vdots \\ x_n \end{pmatrix}= \begin{pmatrix} Qd_{1} \\ Qd_{2} \\ \vdots \\ Qd_{n} \end{pmatrix} 
\end{equation*}

Using Cramer's rule, we have that for each $1 \leq i \leq n$,

\begin{align*}
x_i  & =\frac{\begin{vmatrix} 
b_{11} & \cdots & Qd_{1}  & \cdots & b_{1n} \\
b_{21} & \cdots & Qd_{2} & \cdots & b_{2n} \\
\vdots & \cdots &  \vdots  & \ddots & \vdots \\
b_{n1} & \cdots & Qd_{n} & \cdots & b_{nn} \end{vmatrix}}{|B|} 
=
\frac{Q \begin{vmatrix} b_{11} & \cdots & d_1 & \cdots & b_{1n} \\
b_{21} & \cdots &  d_2 & \cdots & b_{2n} \\
\vdots  & \cdots  & \vdots & \ddots & \vdots  \\
b_{n1} & \cdots & d_n & \cdots & b_{nn} \end{vmatrix}}{Q} \\
& =\begin{vmatrix} b_{11} & \cdots & d_1 & \cdots & b_{1n} \\
b_{21} & \cdots &  d_2 & \cdots & b_{2n} \\
\vdots  & \cdots  & \vdots & \ddots & \vdots  \\
b_{n1} & \cdots & d_n & \cdots & b_{nn} \end{vmatrix}
\end{align*}

Since each of the entries in this last matrix is in $R$, $x_i \in R$. Hence $QD \in span \left\{ B_1, \cdots, B_n \right\} $. Since $D$ was arbitrary, $QR^n \subset span \left\{B_1,\cdots, B_n \right\}$. 

\end{proof}

If we can find an element $Q$ that divides the determinants of all $n$-element subsets of $R_G$, this will give us a determinantal criterion in one direction. 

\begin{theorem} \label{det-theorem} Let $R$ be an integral domain. Suppose there exists some $Q \in R$ so that $Q|\det \begin{bmatrix} C_1, \ldots C_n \end{bmatrix}$ for any $\left\{C_1, \ldots C_n\right\} \subset R_{G}$. Then if  $B=\left\{B_1, \ldots B_n\right\}$ is any subset of $R_G$ for which $\det (B)= Q$, then $B$ is a basis for $R_{G}$.
\end{theorem}

\begin{proof} 
Suppose that $\det \begin{bmatrix} B_1,\ldots,B_n \end{bmatrix}=Q$. Note that $Q \neq 0$,  so $\left\{ B_1,\ldots,B_n \right\}$ must be linearly indpendent over $R$. By Lemma \ref{Cramer}, $QR^n \subset span\left\{ B_1,\ldots, B_n \right\}$. Let $F \in R_G$. Then $QF \in span \left\{ B_1,\ldots, B_n \right\} $ so $QF=\sum_{i=1}^n s_iB_i$ where each $s_i \in R$. Then 
\begin{align*}
s_iQ  = &  s_i\det[B_1,\ldots,B_n] \\
= & \det[B_1,\ldots,B_{i-1},s_i B_i, B_{i+1},\ldots,B_n] \\
= & \det[B_1, \ldots, B_{i-1}, \sum s_j B_j, B_{i+1},\ldots, B_n] \\
= & \det[B_1,\ldots,B_{i-1}, QF, B_{i+1}, \ldots,B_n]\\
= & Q  \det[B_1,\ldots,B_{i-1}, F, B_{i+1}, \ldots, B_n] 
\end{align*}

Since by hypothesis $Q$ divides the determinant of any set of $n$ elements in $R_G$, 

$ Q  \det[B_1,\ldots,B_{i-1}, F, B_{i+1}, \ldots, B_n] = Q^2r$, for some $r \in R$. Then $s_iQ = Q^2r$ implies $s_i = Qr$, so $Q$ divides $s_i$, and this is true for each $i$.
  
  \

Thus $F=\sum(s_i/Q)B_i \in span \left\{B_1,\ldots,B_n \right\}$. So $\mathcal{B}$ spans $R_G$, and hence is a basis.
\end{proof}

When $R$ is a GCD domain over an arbitrary graph $G$, we can always find a non-unit $Q$ satifying the hypotheses of Theorem \ref{det-theorem}. Before we prove this lemma, it is worth noting that the definitions of a $\gcd$ and an $\lcm$ of two elements in a GCD domain are slightly different from the usual defintions in $\mathbb Z$, in that $\gcd$'s and $\lcm$'s are not unique, and the notion of a smallest or largest element doesn't have an analog in an arbitrary GCD domain.

\begin{definition}
Let R be a GCD domain. Then
$d$ is a $gcd$ of $a$ and $b$ if and only if the following two conditions hold:
\begin{enumerate} 
\item 
$d|a$ and $d|b$ 
\item if $c|a$ and $c|b$, then $c|d$. 
\end{enumerate}
\end{definition}

\begin{definition}
Let R be a GCD domain. Then
$l$ is an $lcm$ of $a$ and $b$ if and only if the following two conditions hold:
\begin{enumerate} 
\item 
$a|l$ and $b|l$ 
\item if $a|c$ and $b|c$, then $l|c$. 
\end{enumerate}
\end{definition} 

It follows from this definition that in a GCD domain $R$, if $d_1$ and $d_2$ are gcds of $a$ and $b$, then there exists a unit $u$ for which $ud_1=d_2$, and similarly for lcms. 

Furthermore, given $a,b \in R$ and $d$ a gcd of $a$ and $b$, an lcm $l$ of $a$ and $b$ can be chosen so that $dl = ab$. Equivalently, if $a$ and $b$ are nonzero elements and $d$ is any gcd $d$ of $a$ and $b$, then $ab/d$ is an lcm of $a$ and $b$. 

Throughout the rest of this paper, we will use the notation $\gcd(a,b)$ and $\lcm(a,b)$ with the understanding that they are only defined up to multiplication by units. Furthermore, given a gcd of $a$ and $b$, we will always choose an lcm so that $\gcd(a,b)=\frac{ab}{\lcm(a,b)}$. 

Most basic facts about gcd's and lcm's still hold in GCD domains, but these results must be proved without assuming either Bezout's Theorem or unique factorization. We provide proofs here for completeness.

\begin{lemma} \label{gcd(ax,bx)}Let $R$ be a GCD domain and let $a,b, x \in R$. Then $\gcd(ax,bx)=x\gcd(a,b)$. \end{lemma} 

\begin{proof}
 Let $g=\gcd(a,b)$ and $h=\gcd(ax,bx)$. Since $g|a$ and $g|b$,  $gx|ax$ and $gx|bx$. By definition, $gx|\gcd(ax,bx)=h$. Thus there exists some $y \in R$ so that $gxy=h$. Since $h=\gcd(ax,bx)$, $h|ax$. Thus $gxy|ax$ and so $gy|a$. A similar argument shows that $gy|b$. Since $g=\gcd(a,b)$, $gy|g$ and so $y$ is a unit. But then $\gcd(a,b)xy=gxy=h=\gcd(ax,bx)$, so $\gcd(a,b)=\gcd(ax,bx)$. 
\end{proof}

\begin{lemma} \label{a|bc} Let $R$ be a GCD domain. Then for any $a,b,c \in R$, if $a|bc$ and $\gcd(a,b)=1$, then $a|c$ \end{lemma} 
\begin{proof} 
Suppose that $a|bc$ and that $\gcd(a,b)=1$. Then there exists some $x \in R$ so that $ax=bc$. If we let $g=\gcd(b,x)$, then $g|b$ and there exists a $y \in R$ for which $gy=b$.  
By Lemma \ref{gcd(ax,bx)},  $ag=a\gcd(b,x)=\gcd(ab,ax)$. But $\gcd(ab,ax) =\gcd(ab,bc)=b\gcd(a,c)=dy\gcd(a,c)$. Thus $a=y\gcd(a,c)$. By defintion, $y|a$ and since $dy=b$, $y|b$. Together, these imply that $y|\gcd(a,b)$ and since $\gcd(a,b)=1$, $y$ must be a unit. Since $a$ and $\gcd(a,c)$ differ by a unit $a|c$.  
\end{proof} 

\begin{lemma} Let $R$ be a GCD domain and suppose that $a_1,\ldots,a_n \in R$ are pairwise relatively prime. Then so are $a_1^m,\cdots, a_n^m$ for any integer $m \geq 1$. \end{lemma} 

\begin{proof}

We prove this in the case where $n=m=2$. A similar argument  will give the more general case. Suppose that $\gcd(a,b)=1$. We'll show that this implies $\gcd(a^2,b^2)=1$. Let $g=\gcd(a^2,b^2)$. Then there exist $x,y \in R$ so that $gx=a^2$ and $gy=b^2$. Thus $gxy=a^2y=b^2x$ and so $a|b^2x$.  Since $\gcd(a,b)=1$, Lemma \ref{a|bc} implies $a|x$, and so there exists some $z \in R$ so that $az=x$.  Thus $a^2z^2=x^2$, which implies $gxz^2=x^2$, and so $gz^2=x$.  But this implies that $az=gz^2$ and so $a=gz$.  Thus $g|a$.  A similar argument shows that $g|b$, and so $g|\gcd(a,b)$.  Since $\gcd(a,b)=1$, $g|1$ and so $\gcd(a^2,b^2)=1$, as desired. 
\end{proof}

\begin{theorem} \label{lcm-lemma}Let $R$ be a GCD domain and let $(G,A)$ be an edge-labeled graph where $A=\left\{a_1,\ldots,a_m\right\}$. Let $Q =\lcm(a_1, \cdots, a_m)$. Then Q divides the determinant of any set of $n$ elements of $R_G$. In particular, if $a_1,\ldots,a_m$ are pairwise relatively prime, then $Q = a_1\cdots a_m$.
\end{theorem}
\begin{proof}  
Let  $C_i=(c_{i1},\ldots,c_{in}) \in R_G$ and define $C=[C_1,\cdots,C_n]$. Then 

\begin{equation*}
|C|= \begin{vmatrix} c_{11} & c_{12} & \cdots &c_{1n} \\
c_{21} & c_{22} & \cdots & c_{2n} \\
\vdots & \vdots &  \ddots & \vdots\\
c_{n1} & c_{n2} & \cdots & c_{nn} \\ \end{vmatrix} = 
\begin{vmatrix} c_{11} - c_{21} & c_{12}-c_{22} & \cdots & c_{1n}-c_{2n} \\
c_{21} & c_{22} & \cdots & c_{2n} \\
 \vdots & \vdots & \ddots & \vdots \\
c_{n1} & c_{2n} & \cdots & c_{nn} \\ \end{vmatrix} \end{equation*}
Let $a_1$ be the edge label between vertices $v$ and $w$. Without loss of generality, assume $v=v_1$ and $w=v_2$. Then $a_1$ divides the difference between the vertex labels associated to $v_1$ and $v_2$. This means  $a_1|c_{1j}-c_{2j}$ for all $1 \leq j \leq n$, so $a_1||C|$. A similar argument shows that for any $k$,  $a_k | |C|$. This implies that $Q=\lcm(a_1,\cdots,a_n)||C|$. If $a_1, \ldots, a_m$ are pairwise relatively prime, $Q=\lcm(a_1, \ldots, a_m)=a_1 \cdots a_m | |C|$. 
\end{proof}

We have an easy partial converse to Theorem \ref{det-theorem} .

\begin{proposition} Let $R$ be an integral domain and let $\mathcal{B}=\left\{ B_1,\ldots,B_n \right\}$ be a basis for $R_G$. Then if $Q = \det(\mathcal{B})$, $Q|\det(C)$ for any n-element subset $C=\left\{ C_1, \dots, C_n\right\}$ of $R_G$.
\end{proposition}

\begin{proof}
Since $B_1, \ldots, B_n$ form a basis for $R_G$, each $C_i$ is a linear combination of the $B_j$'. Thus we may write 
$C_i = \sum_{j=1}^n a_{ij}B_j$ for some $a_{ij} \in R$. Then if $A=\begin{bmatrix} A_ij \end{bmatrix}$ is the  matrix whose $ij$th entry is $a_ij$, we have that 

\begin{equation*}
\begin{bmatrix} C_1, \ldots C_n \end{bmatrix} =A\begin{bmatrix} B_1, \ldots, B_n \end{bmatrix} 
\end{equation*}

Thus $\det\begin{bmatrix} B_1, \ldots , B_n \end{bmatrix}|\det \begin{bmatrix} C_1, \dots, C_n \end{bmatrix}$, as desired. 

\end{proof}

We'd like a full converse to Theorem \ref{det-theorem}, but in order to do so we need to find a $Q$ that is maximal with respect to dividing the determinants of all $n$-element subsets of $R_G$. When the edge labels are relatively prime, the desired $Q$ will be their product. We prove this result below, generalizing Theorem 2.3 from \cite{multivariatesplines} to GCD domains and arbitrary graphs. First we need a few lemmas.


\begin{lemma}\label{hat} Let $R$ be a GCD domain and let $a_1,\dots, a_n \in R$ be pairwise relatively prime. Suppose  $s|\hat{a_i} = a_1...a_{i-1}a_{i+1}...a_n$ for all i. Then $s$ is a unit.
\end{lemma} 

\begin{proof}
Since $s$ is a common divisor of the $\hat{a_i}$ , $s|\gcd(\hat{a_1},\ldots,\hat{a_n})$.  
An argument from elementary number theory shows that 

\begin{equation*}
\gcd(\hat{a_1},\cdots,\hat{a_n}) = \frac{a_1\cdots a_n}{lcm(a_1,\ldots,a_n)}
\end{equation*} 

Since the $a_i$'s are pairwise relatively prime,  $\lcm(\hat{a}_1,\cdots,\hat{a}_n) = a_1\cdots a_n$. Then the above equation implies that $gcd(\hat{a}_1, \ldots, \hat{a}_n)=1$, so $s|1$ and $s$ is a  unit or $1$. 
\end{proof} 

We can now prove the main theorem of this section, which states that even when no flow-up class basis exists, we not only have a formula for the determinant of any basis of $R_G$, but we also have a criterion to determine if $R_G$ is free, as long as the edge labels are pairwise relatively prime.

\begin{theorem}  \label{relprime} Let $R$ be a GCD domain, let $G$ be a graph with relatively prime edge labels, and let $Q = a_1 \cdots a_m$. Then $R_{G}$ is free with basis $\mathcal{B}=\{B_1,.,.B_n\}$ if and only if $\det(\mathcal{B}) = uQ$ for some unit $u \in R$.
\end{theorem}

\begin{proof}
Let $\mathcal{B}=\left\{ B_1, \ldots, B_n \right\}$ be a basis for $R_G$. By Theorem \ref{lcm-lemma}, $\det(\mathcal{B})=sQ$ where $Q=a_1\cdots a_m$ and $s \in R$.  Let $Q_{{a}_i} = \hat{a_i}=a_1\cdots a_ia_{i+1}\cdots a_n$. Without loss of generality, let $a_1$ be an edge connecting $v_1$ and $v_2$ in $G$ and observe that 
\begin{equation}
    QQ_{{a}_1}^{n-1} = 
    \begin{vmatrix} 
    Q_{{a}_1} & 0 & 0 & \cdots & 0 \\
    Q_{{a}_1} & Q & 0 & \cdots & 0 \\
    0 & 0 & Q_{{a}_1} & \cdots & 0 \\ 
    \vdots & \vdots & \vdots & \ddots & 0 \\
    0 & 0 & 0 & \cdots &  Q_{{a}_1}  \end{vmatrix}
\end{equation}
Since $a_1$ connects $v_1$ and $v_2$ in $G$, each column in the above matrix is a spline in $R_G$, hence is in the span of the $B_i$'s. Since $\det(\mathcal{B})=sQ$ , the determinant of the above matrix is equal to $Qsr_1$ for some $r_1 \in R$. But then $Qsr_1=Q(Q_{{a}_1})^{n-1}$, so $(Q_{{a}_1})^{n-1}=sr_1$. Thus $s|(Q_{{a}_1})^{n-1}$. A similar argument can be made for any edge $a_i$, so in fact we have that $s|(Q_{{a}_i})^{n-1}$ for all edge labels $a_i$ and so $s|\hat{a_i}$ for all $i$. Since the $a_i$ are pairwise relatively prime, by Lemma \ref{hat}, $s$ is a unit.  Thus $\det(\mathcal{B})=sQ$, where $s$ is a unit, as desired. 
Conversely, suppose $\det(\mathcal{B}) = Q$. 
By Lemma \ref{lcm-lemma},  Q divides the determinant of any set of $n$ elements of $R_G$, hence by Theorem \ref{det-theorem}, $\mathcal{B}$ is a basis for $R_G$.
\end{proof}

We conjecture that when the edge labels are not relatively prime, the previous theorem will be true when $Q$ is maximal with respect to dividing all determinants of n-element subsets of $R_G$.  In \cite{jefflauren}, Rose and Suzuki constructed explicit flow-up class bases when $R = \mathbb{Z}$ - but the proofs extend easily to any PID - where the leading terms are described by least common multiples and greatest common divisors of certain subsets of the edge labels. 

Let $(G,A)$ be an edge-labeled graph where $A = \{a_1,\ldots,a_m\}$ is an arbitrary set of elements of a GCD domain R.  Let $\{l_1,\ldots, l_n\}$ have the same theoretical form (in terms of gcds and lcms of edge labels) as the leading terms of the flow-up class bases constructed in \cite{jefflauren}. 

\textbf{Conjecture}. Let $R$ be a GCD domain and $Q$ be the product of the leading terms $l_i$ defined above. Then $B$ is a basis for $R_G$ if and only if $\det(B) = uQ$ for some unit $u \in R$.  

\begin{remark}
Gjoni \cite{gjoni}, Mahdavi \cite{mahdavi}, and Blaine \cite{blaine} proved this conjecture for integer splines on cycles and diamond graphs, and Altinok and Sarioglan (\cite{alt1}, \cite{alt2}) extended their results to arbitrary GCD domains in the special case of cycles and diamond graphs.     
\end{remark}

\section{Spline Modules with no Flow-up Class Bases}

The theorem below illustrates that Theorem \ref{fcbasis} is likely the best possible, i.e. if $R$ is not a PID, then a flow-up class basis is not guaranteed.

\begin{figure}
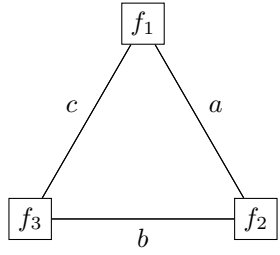

$\kthree{0.3}{f_{1}}{f_{2}}{f_{3}}{a}{b}{c}$
\caption{Generalized spline on a 3-cycle with edge weights $a,b,c$ }
\label{no pid no fc basis}
\end{figure}

\begin{theorem} Let $R$ be a GCD domain that is not a PID. Then there always exists a spline module that does not have a flow-up class basis. \end{theorem}

\begin{proof} 
We prove this by finding an explicit graph and edge labels for which $R_G$ has no flow-up class basis. Let $G$ be a three-cycle with edge labels $a,b,c \in R$ as in Figure \ref{no pid no fc basis}. 

Suppose that $a, b, c $ are pairwise relatively prime, where $<b,c>$ is not principal and  $a \notin <b,c>$.

The defining equations for the spline in Figure \ref{no pid no fc basis} imply that  $a|f_1-f_2$, $b|f_2-f_3$ and $c|f_3-f_1$. Thus, any flow up class basis has the form  $B_1=(1,1,1)$, $B_2=(0,ax,cy)$ and $B_3=(0,0,bcz)$ for some $x,y,z$ in $R$. Let $B=[B_1,B_2,B_3]$. Then $\det(B)=abcxz$. On the other hand, since $a,b,c$ are pairwise relatively prime, $\det(B)=uabc$ for some unit $u \in R$. Thus, $abcxz=uabc$ so $xzu^{-1}=1$ and so $x$ and $z$ are also units in $R$. Since $b|f_2-f_3$ where $f_2=ax$ and $f_3=cy$, we have that there exists some $w \in R$ so that $ax-cy=bw$. Since $x$ is a unit, we may divide both sides by $x$ to obtain $a=x^{-1}cy + x^{-1}bw$ so that $a \in <b,c>$. This is a contradiction. Thus no flow-up class basis for $R_G$ exists. 
\end{proof} 

\begin{figure}
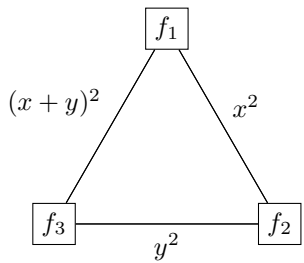

$\kthree{0.3}{f_{1}}{f_{2}}{f_{3}}{x^2}{y^2}{(x+y)^2}$
\caption{Generalized spline on a 3-cycle with edge weights $x^2,y^2,(x+y)^2$ }
\label{no fc basis polynomial}
\end{figure}

\begin{proposition}: Let $G$ be a three-cycle with edge labels $x^2,y^2$ and $(x+y)^2$ over the ring $R=k[x,y]$ where $\text{char}(k) \ne 2.$ Then $R_G$ is free, but does not have a flow-up class basis.
\end{proposition}
.
\begin{proof} First note that $R_G$ is free by \cite{rose2}, since it corresponds to a polynomial spline module in two variables over a subdivided disk. Assume it has a flow-up class basis $\mathcal{B}=\{ B_1, B_2, B_3 \}$. Then by Theorem \ref{relprime}, $\det(\mathcal{B})= ux^2y^2(x+y)^2$ for some unit $u \in k[x,y]$. Then we can assume that $B_1 = (1,1,1)$, $B_2=(0,f,g)$ and $B_3 = (0,0,h)$. Referring to Figure \ref{no fc basis polynomial}, the defining equations imply $f=rx^2$, $g=s(x+y)^2$, $g-f=ty^2$, and $h=py^2(x+y)^2$, for some $r,s,t,p \in k[x,y]$.

Then $\det[B_1, B_2, B_3]=fh=rpx^2y^2(x+y)^2$. Since this must also equal $ux^2y^2(x+y)^2$ , we have that $rpu^{-1}=1$, so $r$ and $p$ are units, hence nonzero constants. Since $g-f=ty^2$, we have that $s(x+y)^2 = rx^2+ty^2$. Equating coefficients, we see that $s=t=r$. Hence $(x+y)^2 = x^2+y^2$, which is false as long as $\text{char}(k) \ne 2.$
\end{proof} 

Although they are not guaranteed, flow-up class bases can exist over $R=k[x,y]$. A trivial example would be if $G$ is a tree, since in this case $R_G$ has a flow-up class basis over any commutative ring with unity. (\cite{julia}.) However, there are non-trivial examples, as we see below.

\begin{figure}
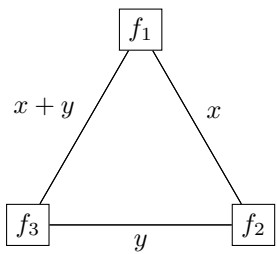

$\kthree{0.3}{f_{1}}{f_{2}}{f_{3}}{x}{y}{x+y}$
\caption{Generalized spline on a 3-cycle with edge weights (x,y,x+y)}
\label{xyspline} 
\end{figure}

\begin{example}
Let $G$ be the 3-cycle graph in Figure \ref{xyspline} with edge labels $x$, $y$, and $x+y$. Let $B_1 =(1,1,1), B_2=(0,x,x+y)$ and $B_3=(0,0,y(x+y))$.  It's easy to see that these satisfy the spline conditions, and that $\det[B_1,B_2,B_3]=xy(x+y)$, the product of the edge labels. It follows from Theorem \ref{det-theorem}, that $B_1, B_2,B_3$ form a flow-up class basis for $R_G$.
\end{example}

\bibliographystyle{amsplain}

\end{document}